\documentclass[reqno]{amsproc}
\usepackage[utf8]{inputenc}
\usepackage{amsfonts}
\usepackage{graphicx}
\usepackage{amscd}
\usepackage{amsmath}
\usepackage{amssymb}
\usepackage{latexsym}
\usepackage{hyperref}
\usepackage[all]{xy}
\usepackage{color}
\usepackage{mathrsfs}
\theoremstyle{plain}

\newtheorem{corollary}{\bf Corollary}

\newtheorem{lemma}{\bf Lemma}

\newtheorem{remark}{\bf Remark}

\newtheorem{theorem}{\bf Theorem}
\numberwithin{equation}{section}

\begin{document}

\title[A note on Serrin's type problem on Riemannian manifolds
]{A note on Serrin's type problem on Riemannian manifolds}

\author[Allan Freitas]{Allan Freitas}
\address{Departamento de Matemática, Universidade Federal da Paraíba, Cidade Universitária, 58051-900, João Pessoa, Paraíba, Brazil.}
\email{allan@mat.ufpb.br, allan.freitas@academico.ufpb.br}

\author[Alberto Roncoroni]{Alberto Roncoroni}
\address{Dipartimento di Matematica, Politecnico di Milano, Piazza Leonardo da Vinci 32, 20133, Milano, Italy.}
\email{alberto.roncoroni@polimi.it.}

\author[Márcio Santos]{Márcio Santos}
\address{Departamento de Matemática, Universidade Federal da Paraíba, Cidade Universitária, 58051-900, João Pessoa, Paraíba, Brazil.}
\email{marcio.santos@academico.ufpb.br}


\keywords{Overdetermined PDE; Conformal fields; Rigidity}
\subjclass[2020]{Primary 35R01, 35N25, 53C24; Secondary 35B50, 58J05, 58J32.}

\begin{abstract}In this paper, we deal with Serrin-type problems in Riemannian manifolds. First, we obtain a Heintze-Karcher inequality and a Soap Bubble result, with its respective rigidity, when the ambient space has a Ricci tensor bounded below. After, we approach a Serrin problem in bounded domains of manifolds endowed with a closed conformal vector field. Our primary tool, in this case, is a new Pohozaev identity, which depends on the scalar curvature of the manifold. Applications involve Einstein and constant scalar curvature spaces.
\end{abstract}
\maketitle

\section{Introduction}

In \cite{serrin}, inspired by a question proposed by Fosdick of a problem related to fluids dynamics, Serrin has begun to study the rigidity of the following Poisson equation with boundary constraints:
\begin{eqnarray}
\Delta u  &=& -1 \ \ \hbox{in}\quad \Omega, \quad u=0\quad \hbox{on}\quad  \partial\Omega, \label{serrinok1}\\
u_{\nu}&=&-c\quad \hbox{on}\quad  \partial\Omega,\label{serrinok2}    
\end{eqnarray}
where $\Omega\subset\mathbb{R}^n$ is a bounded domain, $\nu$ is the outward unit normal, $u_\nu$ is the normal derivative of $u$ on $\partial\Omega$ and $c$ is a positive constant. This underlying work proved that the problem \eqref{serrinok1}-\eqref{serrinok2} admits a solution if and only if $\Omega$ is a ball and $u$ is a radial function. The technique used to solve this problem relies on the so-called \emph{method of moving planes}. We refer to the original paper \cite{serrin} for further details. In the same edition of the above-mentioned Serrin's work, Weinberger \cite{wein} provided a simpler proof based on integral identities; this method is nowadays known as the {\it{$P$-function approach}}. In short, Weinberger defines a sub-harmonic function $P(u)$ associated with the solution of \eqref{serrinok1}-\eqref{serrinok2} and applies the classical strong maximum principle to prove that this function is constant, which, in terms, implies the rigidity result. One main step to prove this constancy is the classical Pohozaev identity for Euclidean domains (we refer to Remark \ref{remark2} below and to the paper \cite{NT} for further details). After Serrin and Weinberger's papers in the literature, there are several alternative proofs and generalizations of Serrin's result in $\mathbb{R}^n$ (see, e.g., \cite{BC, BNST, CGS, CH, CS, FJ, FK, FV, FGK, GL, Payne, SS, Sperb} and the references therein). 

It is no coincidence that these two methods to prove the rigidity of the problem \eqref{serrinok1}-\eqref{serrinok2}, i.e., the moving planes method and the integral technique, remember two classical proofs for the {\it Alexandrov's Soap Bubble Theorem} that claims the only compact embedded constant mean curvature hypersurface in the Euclidean space is a sphere. Moreover, as shown in \cite{poggesi} (see also Corollary \ref{corolario4} below), there is an interesting connection between the mean curvature of $\partial\Omega$ and $u_{\nu}$, where $u$ solves \eqref{serrinok1}-\eqref{serrinok2}.

In particular, in the proof by Ros in \cite{ros} of Alexandrov's theorem via integral inequalities, one crucial step is the \emph{Heintze-Karcher inequality} that reads as follows: given an $n-$dimensional Riemannian manifold $(M,g)$ with non-negative Ricci curvature and given a bounded domain $\Omega\subset M$ such that the mean curvature $H$ of $\partial\Omega$ is positive, then  
\begin{equation}\label{HK_Ros}
 \frac{n-1}{n}\int_{\partial\Omega}\frac{1}{H}\geq Vol(\Omega).
\end{equation}
Moreover, the equality holds in \eqref{HK_Ros} if and only if $\Omega$ is isometric to the Euclidean ball. The proof of \eqref{HK_Ros} uses the Reilly identity proved in \cite{reilly} (see also \eqref{reilly} below), applied to the solution of \eqref{serrinok1}. A natural question is whether or not the Heintze-Karcher inequality can be extended in the class of $n-$dimensional Riemannian manifolds $(M,g)$ satisfying the following lower bound on the Ricci curvature
\begin{equation}\label{Ricci}
   Ric\geq (n-1)kg \,, \quad \text{for some $k\in\mathbb{R}$}.
\end{equation}
By considering a (positive) solution to the following problem 
\begin{eqnarray}\label{serrinok3}
\Delta u+nku  &=& -1 \ \ \hbox{in}\quad \Omega, \quad u=0\quad \hbox{on}\quad  \partial\Omega, 
\end{eqnarray}
where $\Omega\subset M$ is a bounded domain, we obtain our first result 

\begin{theorem}\label{teoA}
    Let $(M,g)$ be an $n-$dimensional Riemannian manifold satisfying \eqref{Ricci}. Let $\Omega\subset M$ be a bounded domain such that the mean curvature of $\partial\Omega$ is positive and $u$ be a positive solution of \eqref{serrinok3}. Then
\begin{equation*}
 \frac{n-1}{n}\int_{\partial\Omega}\frac{1}{H}\geq Vol(\Omega)+nk\int_{\Omega}u, 
\end{equation*}
and the equality occurs if and only if $\Omega$ is a geodesic ball and $u$ is a radial function. 
\end{theorem}

We mention that when $k=0$, we recover the classical Heintze-Karcher inequality \eqref{HK_Ros}. The essential tool here is Reilly's identity applied to the solution of \eqref{serrinok3} and to perform the analog of Ros' proof. In particular, as a Corollary of this Theorem, we obtain an Alexandrov Soap Bubble-type result (see Theorem \ref{soapbubble} below).

Talking about Weiberger's $P$-function approach, the two ingredients associated with his technique, that is, a maximum principle for a suitable $P$-function and a Pohozaev-type identity, have been reproduced to study overdetermined problems for domains in Riemannian manifolds. In particular, by studying Serrin's problem on simply connected manifolds with constant sectional curvature, the space forms, Ciraolo and Vezzoni \cite{ciraolo} used the $P$-function approach to prove the analog of Serrin's theorem for the problem \eqref{serrinok3} with the overdetermined condition
\begin{eqnarray}
u_{\nu}&=&-c\quad \hbox{on}\quad  \partial\Omega,\label{serrinok4}    
\end{eqnarray}
where $\Omega\subset M$ is a bounded domain and $M=\mathbb{R}^n$ when $k=0$, $M=\mathbb{H}^n$ when $k=-1$ and $M=\mathbb{S}^n_+$ when $k=1$. We also refer to \cite{kumpraj} for the same result via the method of moving planes. Following the line in \cite{ciraolo}, Farina and the second author \cite{farina-roncoroni} (see also \cite{Roncoroni} for a previous partial result) have studied this same problem on warped products, also obtaining rigidity results by considering such ambients with Ricci curvature bounded below. 
 On the other hand, it is known that Serrin’s theorem is, in general false; indeed, given a compact Riemannian manifold $(M,g)$ such that there exists a non-degenerate critical point $p\in M$ of its scalar curvature function, then it is possible to construct a smooth foliation $(\partial\Omega_{\varepsilon})$ of a neighborhood of $p$, where $\Omega_{\varepsilon}$ is a domain in which \eqref{serrinok1}-\eqref{serrinok2} possesses a solution with $c=\frac{\varepsilon}{n}$ (we refer to \cite{fall}, and to \cite{pacard} and \cite{delay} for related results for the eigenvalue problem). Moreover, we mention that the existence results for overdetermined problems for other semilinear equations in compact Riemannian manifolds have been proved in \cite{dv1} and \cite{dv2}. 


In this work, we deal with Serrin-type problems on Riemannian manifolds, providing geometric conditions on the ambient space to obtain rigidity results. In particular, we study problem \eqref{serrinok3}-\eqref{serrinok4} in a class of Riemannian manifolds satisfying \eqref{Ricci} and that admit a closed conformal vector field. A vector field $X\in\mathfrak{X}(M)$ is called \textit{closed conformal} if
\begin{equation*}
\nabla_Y X=\varphi Y,   \quad \text{ for all $Y\in\mathfrak{X}(M),$} 
\end{equation*}
for some smooth function $\varphi$ called the \emph{conformal factor}. We mention that manifolds endowed with a nontrivial closed conformal vector field are locally isometric to a warped product with a 1-dimensional factor (for details, we refer, e.g., to ~\cite[Section~3]{montiel}). In this sense, the results of this paper extend and could be compared with those previously obtained in \cite{farina-roncoroni}. Despite this, we provide new geometric identities and conclusions (even in the case of warped products) that permit their applications in Einstein and constant scalar curvature ambient spaces. For such a class, we obtain a Pohozaev identity that we think it posses its own interest (see Lemma \ref{lemmadif} below). This identity will be combined with the $P$-function approach to obtain geometric constraints that imply rigidity (see Theorem \ref{mainthm1} below). Some interesting applications are noticed when the ambient is, e.g., an Einstein manifold.

\begin{theorem}\label{teoB}
Let $(M,g)$ be an Einstein manifold with $ Ric= (n-1)kg$, for some $k\in \mathbb{R}$ such that $M$ is endowed with a closed conformal field with a positive conformal factor $\varphi$. Let $u$ be a positive solution of the problem \eqref{serrinok3}-\eqref{serrinok4}, then $\Omega$ is a metric ball, and $u$ is a radial function.    
\end{theorem}

Finally, another important tool in Ros' proof of the Alexandrov Soap bubble Theorem (besides the already mentioned Heintze-Karcher inequality) is the \emph{Minkowski identity} (see e.g. \cite{reilly2}): 

\begin{equation*}\label{Mink}
\int_{\partial\Omega} H\langle x-p, \nu\rangle dx= (n-1)|\partial\Omega|, 
\end{equation*}
where $p\in\mathbb{R}^n$ and  $\Omega$ is a bounded domain in $\mathbb{R}^n.$ Motivated by this we have the following result

\begin{theorem}\label{teoC}
Let $(M,g)$ be a manifold endowed with a closed conformal vector field $X$ and constant scalar curvature $R=n(n-1)k.$ Suppose that $u$ is a positive solution of the problem \eqref{serrinok3}-\eqref{serrinok4}. Then, $$\int_{\partial\Omega}\langle X,\nu\rangle ((n-1)-cnH)=0,$$ 
where $H$ denotes the mean curvature of $\partial\Omega.$ In particular, if $k=0$, $\langle X,\nu\rangle>0$ and  $\partial\Omega $ has constant mean curvature, then $H=\frac{(n-1)}{n}\frac{|\partial\Omega|}{|\Omega|}.$    
\end{theorem}

\medskip 

\textbf{Organization of the paper.} In Section \ref{HK}, we prove Theorem \ref{teoA} and apply it to establish an Alexandrov Soap Bubble Theorem. In Section \ref{section3}, we consider the class of Riemannian manifolds that admits a closed conformal vector field; we prove a Pohozaev-type identity and obtain a general integral condition that implies the rigidity, giving some applications (like Theorem \ref{teoB} and \ref{teoC}).

\section{Heintze-Karcher inequality and Soap Bubble Theorem}\label{HK}

We begin this section remembering the well-known \emph{Reilly identity}, proved in \cite{reilly}, 

\begin{multline}
\int_{\Omega}^{} \Big[ \frac{n-1}{n}(\Delta f)^2 - | \mathring{\nabla}^2f |^2 \Big] \; \label{reilly}\\      = \int_{\partial \Omega}^{} \Big( h(\bar{\nabla} z, \bar{\nabla} z) + 2 f_{\nu} \bar{\Delta} z + H f_{\nu}^2\Big)  + \int_{\Omega}^{} Ric(\nabla f, \nabla f) ,
\end{multline}     
which holds true for every domain $\Omega$ in a Riemannian manifold $(M^n,g)$ and for every $f \in C^{\infty}(\overline{\Omega})$, where $\mathring{\nabla}^2 f$ denotes the traceless Hessian of $f$, explicitly
$$
\mathring{\nabla}^2 f=\nabla^2 f - \frac{\Delta f}{n} g\, ,
$$
$\bar{\nabla}$ and $\bar{\Delta}$ indicate the gradient and the Laplacian of the induced metric in $\partial\Omega$, $z=f\vert_{\partial\Omega}$ 
and $\nu$ be the unit outward normal of $\partial\Omega$, $h(X,Y)= g(\nabla_{X}\nu, Y)$ and $H=tr_{g}h$ the second fundamental form and the mean curvature (with respect to $\nu$) of $\partial\Omega$, respectively. 

In this section, we consider the following problem
\begin{equation}\label{serrin2}
\left\{\begin{array}{rcl}
\Delta u +nku &=& -1\quad \hbox{in}\quad  \Omega\\
u&>& 0 \\
u&=& 0\quad \hbox{on}\quad  \partial\Omega,\\
\end{array}\right.
\end{equation}
where $k\in\mathbb{R}$ and the main purpose is to obtain the Heintze-Karcher type inequality announced in Theorem \ref{teoA}. The first step is to apply the Reilly identity \eqref{reilly} to the solution of \eqref{serrin2}
\begin{lemma} \label{lemmareilly}
Let $u$ be a solution of \eqref{serrin2}. Then   
\begin{equation}\label{serrin3}  \int_{\Omega}{| \mathring{\nabla}^2 u}|^2+\int_{\Omega}[\mbox{Ric}-(n-1)kg](\nabla u,\nabla u)=-\frac{1}{n}\int_{\partial \Omega}u_{\nu}[(n-1)+nHu_{\nu}]
\end{equation}
\begin{proof}
Reilly's identity applied to the solution of \eqref{serrin2} implies
\begin{equation}\label{eq1}
 \int_{\Omega}|\mathring{\nabla}^2 u|^2=-\int_{\partial\Omega}Hu_{\nu}^2-\int_{\Omega}\mbox{Ric}(\nabla u,\nabla u)+\frac{n-1}{n}\int_{\Omega}(\Delta u)^2.   
\end{equation}
On the other hand, again using \eqref{serrin2},
\begin{eqnarray}
\int_{\Omega}(\Delta u)^2&=&\int_{\Omega}\Delta u (-1-nku)\label{eq2}\\
&=&-\int_{\partial\Omega}u_{\nu}-nk\int_{\Omega}u\Delta u\nonumber\\
&=&-\int_{\partial\Omega}u_{\nu}+nk\int_{\Omega}|\nabla u|^2.\nonumber
\end{eqnarray}
Replacing \eqref{eq2} in \eqref{eq1}, the result follows.
\end{proof}
\end{lemma}

An immediate consequence of the previous formula is the following

\begin{corollary}\label{meanc}
Let $(M^n,g)$ a manifold with $\mbox{Ric}\geq (n-1)kg$, for some $k\in\mathbb{R}$. Let $\Omega\subset M$ be a domain and $u$ a solution of \eqref{serrin2}. If 
\begin{equation}\label{overdet}
u_{\nu}(x)=-\frac{n-1}{nH(x)} \quad \text{on } \partial\Omega,
\end{equation}
then $\Omega$ is a ball and $u$ is a radial function.
\end{corollary}

\begin{proof}
From \eqref{overdet} and the condition on the Ricci curvature, we immediately get, from Lemma \ref{lemmareilly}, that 
$$
\mathring{\nabla}^2 u=0 \quad \text{in $\Omega$}.
$$ 
The conclusion follows immediately from the Obata-type theorem in ~\cite[Lemma~6]{farina-roncoroni}.
\end{proof}

To our understanding, we note that the converse statement of the previous Corollary also holds true in the case of space forms. Indeed, the metric ball posses a radial solution which naturally satisfies $(2.6)$ (for details see ~\cite[Theorem~4]{farina-roncoroni} and, about the relation between the normal derivative and the mean curvature, see Corollary \ref{corolario4}.

\begin{remark}
The technique introduced in this section, what is 
a rearrangement of a Reilly(-type) formula applied to a solution of a PDE can be extended to study similar problems. In this sense, for example, the last result can be related (and indicate directions to extension) to a recent work associated with Serrin's problem to the $p$-Laplacian (see ~\cite[Theorem~1.6]{ruan}).   
\end{remark}

Another consequence of Lemma \ref{lemmareilly} is Theorem \ref{teoA}. 

\begin{proof}[Proof of Theorem \ref{teoA}]
Since
 \begin{eqnarray*}
     \frac{1}{nH}[(n-1)+nHu_{\nu}]^2&=&\frac{(n-1)^2}{nH}+u_{\nu}[(n-1)+nHu_{\nu}]+u_{\nu}(n-1),
 \end{eqnarray*}   
we have
\begin{eqnarray*}
-\frac{1}{n}\int_{\partial \Omega}u_{\nu}[(n-1)+nHu_{\nu}]&=&-\int_{\partial\Omega}\frac{1}{n^2 H}[(n-1)+nHu_{\nu}]^2\\
& &+\left(\frac{n-1}{n}\right)^2\int_{\partial\Omega}\frac{1}{H}+\frac{n-1}{n}\int_{\partial\Omega}u_{\nu}.    
\end{eqnarray*}
On the other hand,
\begin{eqnarray*}
\int_{\partial\Omega}u_{\nu}&=&\int_{\Omega}\Delta u\\
 &=&-Vol(\Omega)-nk\int_{\Omega}u.
\end{eqnarray*}
Then, by Lemma \ref{lemmareilly},
\begin{eqnarray}
& &
 \int_{\Omega}|\mathring{\nabla}^2 u|^2+\int_{\Omega}[\mbox{Ric}-(n-1)kg](\nabla u,\nabla u)+\frac{1}{n^2}\int_{\partial\Omega}\frac{1}{H}[(n-1)+nHu_{\nu}]^2\nonumber\\
 &=&\left(\frac{n-1}{n}\right)^2\int_{\partial\Omega}\frac{1}{H}-\frac{n-1}{n}\mbox{Vol}(\Omega)-(n-1)k\int_{\Omega}u.\label{reillyapplied}
\end{eqnarray}
Since the left-hand side of \eqref{reillyapplied} is non-negative, we obtain the desired inequality. Furthermore, the equality holds if and only if all integrals in the left-hand side of \eqref{reillyapplied} vanish (since each integral is non-negative). In particular, this implies $\mathring{\nabla}^2 u=0$; therefore, $\Omega$ is a geodesic ball, and $u$ is a radial function. 
\end{proof}

The last result that we prove in this section is an integral identity inspired by the ones in \cite{poggesi} and is the following: 

\begin{theorem} [Soap Bubble Theorem]\label{soapbubble}
Let $\Omega$ be a domain in a Riemannian manifold $(M^n,g)$ with $\mbox{Ric}\geq (n-1)kg$, for some $k\in\mathbb{R}$, and $u$ be a solution of \eqref{serrin2}.
Then 
 \begin{equation*}
\int_{\partial\Omega}(H_{0}-H)(u_{\nu})^2\geq 0, 
 \end{equation*}
 where $H$ is the mean curvature of $\partial\Omega$ and $H_{0}=\frac{n-1}{n c}$ with $c$ a constant given by \eqref{choiceR}. In particular, if $H\geq H_{0}$ on $\partial\Omega$ then $\Omega$ is a ball (and, a fortiori, $u_{\nu}=-c$).
\end{theorem}

\begin{proof}
   We start from \eqref{serrin3} and we analyse its right-hand side: 
\begin{eqnarray}
& & -\frac{1}{n}\int_{\partial \Omega}u_{\nu}[(n-1)+nHu_{\nu}]=\nonumber\\
&=&\frac{n-1}{n}\int_{\partial\Omega}u_{\nu}+\frac{(n-1)c}{n}|\partial\Omega|-\frac{1}{c}\int_{\partial\Omega}\left(u_{\nu}+c\right)^2+\int_{\partial\Omega}(H_{0}-H)u_{\nu}^2, \label{reilly3}   
\end{eqnarray}
where we used the following trivial identity
\begin{eqnarray*}
\int_{\partial\Omega}Hu_{\nu}^2&=&H_{0}\int_{\partial\Omega}u_{\nu}^2+\int_{\partial\Omega}(H-H_{0})u_{\nu}^2\\
 &=&\frac{n-1}{nc}\int_{\partial\Omega}\left(u_{\nu}+c\right)^2-2\frac{n-1}{n}\int_{\partial\Omega}u_{\nu}-\frac{(n-1)c}{n}|\partial\Omega|+\int_{\partial\Omega}(H-H_{0})u_{\nu}^2  .
\end{eqnarray*}
We choose 
\begin{eqnarray}
 c&:=&-\frac{\int_{\partial\Omega}u_{\nu}}{|\partial\Omega|}\nonumber\\
 &=&\frac{1}{|\partial\Omega|}\left(Vol(\Omega)+nk\int_{\Omega}u\right)\label{choiceR}
\end{eqnarray}

\end{proof}

\section{Serrin type problem in an ambient endowed with a closed conformal vector field} \label{section3}

Given a Riemannian manifold $(M,g)$, we recall that a vector field $X\in\mathfrak{X}(M)$ is called \textit{closed conformal} if
\begin{equation}\label{conformal}
\nabla_Y X=\varphi Y,    
\end{equation}
for all vector field $Y\in\mathfrak{X}(M),$ where $\varphi$ is a smooth function called conformal factor. Along this section, $M$ denotes a manifold endowed with a closed conformal vector field $X\in \mathfrak{X}(M)$, and $\varphi$ is its conformal factor.

Our first result is the following Bochner-type identity.

\begin{lemma} Let $f$ be a smooth function on $M$. Then,
$$\Delta\langle X,\nabla f\rangle=\langle\nabla\varphi,\nabla f\rangle(2-n)+2\varphi\Delta f+\langle\nabla\Delta f,X\rangle,$$
where $X$ is the closed conformal field with conformal factor $\varphi$.
\end{lemma}
\begin{proof}
First, a straightforward calculation shows that 
\begin{equation}\label{eqlll}
 \nabla \langle X,\nabla f\rangle=\varphi\nabla f+\nabla^2  f(X).   
\end{equation}

We observe that 
\begin{eqnarray*}
 div (\nabla^2 f(X))
 &=&\frac{1}{2}\langle \nabla^2 f, \mathcal{L}_Xg\rangle+\langle\nabla\Delta f,X\rangle +Ric(\nabla f,X),
\end{eqnarray*}
where $\mathcal{L}$ denotes the Lie derivative, and we have used Ricci's identity $div(\nabla^2 f)=\nabla\Delta f+Ric(\nabla f)$. Then, by taking the divergence of \eqref{eqlll},
\begin{equation}\label{laplacian}
    \Delta \langle X,\nabla f\rangle= \langle\nabla\varphi,\nabla f\rangle+\varphi\Delta f+\frac{1}{2}\langle \nabla^2 f,\mathcal{L}_Xg\rangle +\langle\nabla\Delta f,X\rangle +Ric(\nabla f,X)
\end{equation}
On other hand, since \eqref{conformal} the curvature tensor
\begin{eqnarray*}
 R(u,v)X
 &=&\langle u,\nabla \varphi\rangle v-\langle v,\nabla \varphi\rangle u,
\end{eqnarray*}
for all $u,v\in\mathfrak{X}(M)$. Then $Ric(X,\cdot)=-(n-1)\nabla\varphi$. Furthermore, again by \eqref{conformal}, $\mathcal{L}_Xg=2\varphi g.$ Therefore, replacing in \eqref{laplacian} we get the desired result.

\end{proof}

We now devote our attention to the following overdetermined problem
\begin{equation}\label{serrin}
\left\{\begin{array}{rcl}
\Delta u +nku &=& -1 \quad \hbox{in}\quad \Omega\\
u&>&  0\\
u&=& 0\quad \hbox{on}\quad  \partial\Omega,\\
|\nabla u|&=&c\quad \hbox{on}\quad  \partial\Omega,
\end{array}\right.
\end{equation}
where $\Omega$ is a bounded domain in $(M,g)$.

The following result provides a Pohozaev-type identity for such domains. 
\begin{lemma}\label{lemmadif}
    Let $M$ be a manifold endowed with a closed conformal vector field $X.$ Let $u$ be a solution of the problem \eqref{serrin}, then
\begin{equation}\label{Pohozaev}
   \frac{n+2}{n}\int_{\Omega}\varphi u=c^2\int_{\Omega}\varphi-\frac{n-2}{2n(n-1)}\int_{\Omega}u^2(\varphi R+\frac{1}{2}X(R))-2k\int_{\Omega}\varphi u^2 ,
\end{equation}
    where $R$ is the scalar curvature of $M.$
\end{lemma}

\begin{proof}
    If $u$ is a solution of the Serrin problem \eqref{serrin}, Lemma \ref{lemmadif} gives
\begin{equation*}
\Delta \langle X,\nabla u\rangle= \langle\nabla\varphi,\nabla u\rangle(2-n)+2\varphi(-1-nku)-nk\langle X,\nabla u\rangle.
\end{equation*}
Thus, multiplying this identity by $u$ and using again \eqref{serrin}, we conclude that 
\begin{equation}\label{trick}
    u\Delta \langle X,\nabla u\rangle-\langle X,\nabla u\rangle\Delta u=u\langle\nabla\varphi,\nabla u\rangle(2-n)-2\varphi u-2n\varphi k u^2+\langle X,\nabla u\rangle.
\end{equation}
Now, note that $u\Delta \langle X,\nabla u\rangle-\langle X,\nabla u\rangle\Delta u=div(u\nabla \langle X,\nabla u\rangle- \langle X,\nabla u\rangle\nabla u).$ Since $u=0$ along of the boundary and $\nu=-\frac{\nabla u}{|\nabla u|}$, we conclude from divergence theorem 

\begin{equation*}
    \int_{\Omega}u\Delta \langle X,\nabla u\rangle-\langle X,\nabla u\rangle\Delta u=-\int_{\partial\Omega}\langle X,\nabla u\rangle\langle\nabla u,\nu\rangle=c\int_{\partial\Omega}\langle X,\nabla u\rangle
\end{equation*}
Again using the divergence theorem $$\int_{\partial\Omega}\langle X,\nabla u\rangle=-c\int_{\Omega}div X=-cn\int_{\Omega}\varphi.$$
Thus, 
\begin{equation}\label{eqqqq1}
 \int_{\Omega}u\Delta \langle X,\nabla u\rangle-\langle X,\nabla u\rangle\Delta u=-c^2n\int_{\Omega}\varphi. 
\end{equation} 
We intend to study the integral of the right side of \eqref{trick}. First, since $u=0$ on $\partial\Omega$ and $div X=n\varphi$, we have 

\begin{equation}\label{a}
 \int_{\Omega}\langle X,\nabla u\rangle=-\int_{\Omega}udiv X=-n\int_{\Omega}u\varphi.
\end{equation}
Moreover, using again that $u=0$ on $\partial\Omega$,

\begin{equation}\label{b}
    \int_{\Omega}u\langle\nabla \varphi,\nabla u\rangle=-\frac{1}{2}\int_{\Omega}u^2\Delta\varphi.
\end{equation}
 From \eqref{trick}, \eqref{eqqqq1}, \eqref{a}, and \eqref{b} we get that

\begin{equation}\label{prepohozaev}
    \frac{n+2}{n}\int_{\Omega}\varphi u=c^2\int_{\Omega}\varphi+\frac{n-2}{2n}\int_{\Omega}u^2\Delta\varphi-2k\int_{\Omega}\varphi u^2
\end{equation}

Finally, since $Ric(X)=-(n-1)\nabla\varphi,$ we conclude that
\begin{eqnarray}\label{trick2}
 -(n-1)\Delta\varphi&=& div(Ric(X))\\\nonumber
    &=& \frac{1}{2}\langle Ric, \mathcal{L}_Xg\rangle+div(Ric)(X)\\\nonumber
    &=&\varphi R+\frac{1}{2}X(R),
    \end{eqnarray}
where $R$ denotes the scalar curvature of $\Omega.$
From \eqref{prepohozaev} and \eqref{trick2} we get the desired result.

\end{proof}


Now, we consider the following $P$-function
\begin{equation}\label{pfunction2}
P(u)=|\nabla u|^2+\frac{2}{n}u+ku^2.  \end{equation}
If $u$ is a solution of \eqref{serrin}, the Bochner's formula applied to $\nabla u$ can be used to prove (under a condition of Ricci bounded below) the following sub-harmonicity for $P(u)$ as well as a rigidity result (see Farina-Roncoroni~\cite[Lemma~5 and Lemma~6]{farina-roncoroni})
\begin{lemma}[\cite{farina-roncoroni}]\label{farina-roncoroni}
Let $(M^{n},g)$ be an $n$-dimensional Riemannian manifold such that  
$$\mbox{Ric}\geq (n-1)kg.$$
Let $\Omega\subset M$ be a domain and $u\in C^{2}(\Omega)$ a solution of \eqref{serrin} in $\Omega$. Then 
$$\Delta P(u)\geq 0.$$
Furthermore, $\Delta P(u)=0$ if and only if $\Omega$ is a metric ball and $u$ is a radial function.
\end{lemma}

\begin{remark}
\label{remark2}
 These two adapted tools, a Pohozaev-type identity and a suitable maximum principle, make a rule to study Serrin's problem using Weinberger's technique. We remember that to study \eqref{serrinok1}-\eqref{serrinok2} in Euclidean domains, Weinberger defined the $P$-function 
$$
P(u)=|\nabla u|^2+\frac{2}{n}u.
$$ 
This function is subharmonic by 
Cauchy-Schwartz inequality
\begin{equation}\label{pfunction}
 \Delta P(u)=2|\nabla^2 u|^2-\frac{2}{n}\geq 0,  
\end{equation}
and, moreover, 
$$
P(u)=c^2 \quad \text{ on $\partial\Omega$}.
$$
Hence, from the strong maximum principle, one gets that either $P(u)\equiv c^2$ in $\overline{\Omega}$ or $P(u)< c^2$ in $\Omega$. The second case would give a contradiction with the {\it classical Pohozaev identity} applied to \eqref{serrinok1}-\eqref{serrinok2} (if it is necessary, see, e.g., Struwe~\cite[page~171]{struwe} for this identity). This means that $P(u)\equiv c^2$ in $\overline{\Omega}$ and \eqref{pfunction} is equality, which implies that $u$ is a radial function and $\Omega$ is a ball.   
\end{remark}

 From now on, the goal is to provide a condition for the $P$-function associated with the problem to be harmonic. The idea, which follows the original Weinberger's approach, is to use the maximum principle to the sub-harmonic function $P(u)$ to prove that it is constant given an integral condition. The following result, mainly its geometric consequences, gives such a condition. It is directly inspired by the Pohozaev-type identity obtained in Lemma \ref{lemmadif}.

\begin{theorem}\label{mainthm1}
    Let $(M,g)$ be a Riemannian manifold endowed with a closed conformal vector field $X$ such that $Ric\geq (n-1)kg$. If $u$ is a solution of \eqref{serrin}, the conformal factor $\varphi$ of $X$ is positive and 
   \begin{equation}\label{main1}
       \int_{\Omega}u^2\left[\varphi(-R+n(n-1)k)-\frac{1}{2}X(R)\right]\geq 0,
   \end{equation}
    then $\Omega$ is a metric ball and $u$ is a radial function.
    \end{theorem}

\begin{proof}
    Consider the $P$-function given by \eqref{pfunction2}. We observe that $P(u)=c^2$ on $\partial\Omega$. Suppose by contradiction that $P(u)<c^2$ in $\Omega.$
    Thus, 
    \begin{equation}
        \varphi(|\nabla u|^2+\frac{2}{n}u+ku^2)<c^2\varphi.
    \end{equation}
    
From Lemma \ref{lemmadif}, we get that

\begin{equation}\label{A}
    \frac{2}{n}\int_{\Omega}\varphi u=- \int_{\Omega}\varphi u+      c^2\int_{\Omega}\varphi-\frac{n-2}{2n(n-1)}\int_{\Omega}u^2(\varphi R+\frac{1}{2}X(R))-2k\int_{\Omega}\varphi u^2.
\end{equation}
 
On the other hand, taking into account that $u=0$ on the boundary, we use the divergence theorem 

$$\int_{\Omega}\varphi div(u\nabla u)=\frac{1}{2}\int_{\Omega}u^2\Delta\varphi=-\frac{1}{2(n-1)}\int_{\Omega}u^2(\varphi R+\frac{1}{2}X(R)).$$
Moreover, since $\Delta u=-1-nku$ is not hard to see that 
 $$\int_{\Omega}\varphi div(u\nabla u)=\int_{\Omega}\varphi(|\nabla u|^2-u(1+nku)).$$ 
From the above equations, we conclude that 
\begin{equation}\label{B}
  \int_{\Omega}\varphi(|\nabla u|^2)=-\frac{1}{2(n-1)}\int_{\Omega}u^2(\varphi R+\frac{1}{2}X(R))+ \int_{\Omega}\varphi u(1+nku).
\end{equation}
Since  $\varphi(|\nabla u|^2+\frac{2}{n}u+ku^2)<c^2\varphi$, from \eqref{A} and \eqref{B}, we conclude that
$$-\frac{1}{n-1}(\frac{n-2}{2n}+\frac{1}{2})\int_{\Omega}u^2(\varphi R+\frac{1}{2}X(R))+k\int_{\Omega}u^2\varphi(n-1)<0.$$
Thus,

$$\int_{\Omega}u^2\left[\varphi (-R+n(n-1)k)-\frac{1}{2}X(R)\right]<0,$$
what gives a contradiction with \eqref{main1}. Thus, using maximum principle, $P(u)\equiv c^2$ and, consequently, $\Delta P(u)\equiv 0$. The rigidity conclusion follows from Lemma \ref{farina-roncoroni}. 
\end{proof}

In order to give applications of the last result, we look for geometric properties which imply the integral condition \eqref{main1}. For example, the first obvious condition arises in the following result in the context of Einstein manifolds.

\begin{proof}[Proof of Theorem \ref{teoB}]
    Einstein's condition $Ric=(n-1)kg$ implies that $R=n(n-1)k$, $X(R)=0$, and, then \eqref{main1} is satisfied
\end{proof}

This last result can be used to give explicit examples in warped products. We remember that if $M=I\times_{f}N$ is a warped product, the vector field $X=f\partial_t$ is a closed conformal vector field with conformal factor $\varphi=f'$. Thus, we get the following results.

\begin{corollary}
    Let $M$ be a warped product given by $\mathbb{R}\times_{e^t}N,$ where $N$ is a Ric flat manifold. If $u$ is a solution to the problem \ref{serrin}, with $k=-1,$ then $\Omega$ is a metric ball, and $u$ is a radial function.
\end{corollary}

\begin{corollary}
    Let $M$ be a warped product given by $(\varepsilon,+\infty)\times_{\cosh t}N,$ $\varepsilon>0$, where $N$ is an Einstein manifold with scalar curvature given by $\widetilde{R}=-(n-1)(n-2)$. If $u$ is a solution of the problem \ref{serrin}, with $k=-1,$ then $\Omega$ is a metric ball, and $u$ is a radial function.
\end{corollary}

In order to weaken the Einstein manifold assumption, we deal again with the closed conformal property of $X$. First, we notice that, using \eqref{trick2},
 
\begin{eqnarray}   \Delta\varphi+nk\varphi&=&-\frac{1}{n-1}(\varphi R+\frac{1}{2}X(R))+nk\varphi \nonumber\\
&=&\frac{\varphi}{n-1}[-R+nk(n-1)]-\frac{1}{2(n-1)}X(R),\label{laplacphi}
\end{eqnarray}
and, thus, the inequality
    \eqref{main1} is equivalent to
\begin{equation}\label{trick1'}
\int_{\Omega}u^2(\Delta\varphi+nk\varphi)\geq 0.  
\end{equation}

With this in mind, we have the following result
\begin{corollary}
 Let $M$ be a manifold endowed with a closed conformal vector field $X$ satisfying $Ric\geq (n-1)kg.$ In addition, suppose that $u$ is a solution of the problem \ref{serrin} and $Ric(X,\nabla u)\geq (n-1)k\langle\nabla u,X\rangle.$ If the conformal factor $\varphi$ of $X$ is positive then $\Omega$ is a metric ball and $u$ is a radial function. 
\end{corollary}

\begin{proof}
 Since $X$ is a closed conformal vector field,
\begin{equation}\label{trick3}
    -(n-1)u^2\Delta\varphi=u^2div(Ric(X))=div(u^2Ric(X))-\langle \nabla u^2, Ric(X)\rangle
\end{equation}
and
\begin{equation}\label{trick4}
    u^2nk\varphi=u^2kdiv(X)=kdiv(u^2X)-k\langle\nabla u^2, X\rangle
\end{equation}

 Since $u=0$ along boundary, from  \eqref{trick3} and \eqref{trick4} we obtain
$$\int_{\Omega}u^2(\Delta\varphi+nk\varphi)=\frac{2}{n-1}\int_{\Omega}u(Ric(X,\nabla u)-(n-1)k\langle\nabla u,X\rangle).$$
Then the condition $Ric(X,\nabla u)\geq (n-1)k\langle\nabla u, X\rangle$ implies \eqref{trick1'}, and finally we use Theorem \ref{mainthm1}  to guarantee the desired result.
\end{proof}

We note that these series of identities arising from the existence of a closed conformal field also allow dealing with the case of manifolds with constant scalar curvature. The influence of the scalar curvature in the existence of solutions is an exciting topic that, to the knowledge of authors, is treated for the first time in the work of Fall-Millend (\cite{fall}). By considering a non-degenerate critical point $p$ of the scalar curvature function in a Riemannian manifold $(M,g)$, they can construct a smooth foliation $(\partial\Omega_{\varepsilon})$ of a neighborhood of $p$, where $\Omega_{\varepsilon}$ is a domain in which \eqref{serrin} with $k=0$ possesses a solution (here the constant $c=\frac{\varepsilon}{n}$). The following result, which can be viewed as a Minkowski-type identity, provides a necessary condition for having a solution in a case of ambient with constant scalar curvature (and therefore with degenerate critical points of the scalar curvature).
\begin{proof}[Proof of Theorem \ref{teoC}]
   Being $u$ a solution of \eqref{serrin} we get

$$-\int_\Omega\varphi (1+nku)-\int_{\Omega}u\Delta\varphi=\int_{\Omega}div(\varphi\nabla u-u\nabla\varphi)=\int_{\partial\Omega}\varphi\langle\nabla u, \nu\rangle$$
Thus,
$$-\int_{\Omega}u(\Delta\varphi+nk\varphi)=\int_{\partial\Omega}\varphi\langle\nabla u, \nu\rangle+\int_{\Omega}\varphi$$
Now, since $R=n(n-1)k$, we conclude from \eqref{laplacphi}  
$$\int_{\partial\Omega}\varphi\langle\nabla u, \nu\rangle+\int_{\Omega}\varphi=0.
$$
On the other hand, since $\nu=-\frac{\nabla u}{|\nabla u|}$ and $n\varphi= div X$ we get that

\begin{equation*}\label{divergencethe}
-c\int_{\partial\Omega}div(X)+\int_{\partial\Omega}\langle X,\nu\rangle=0.    
\end{equation*}
Now, a straightforward calculation shows that
$\frac{n-1}{n}div(X)=\widetilde{div}(X^T)+H\langle X,\nu\rangle,$
where $\widetilde{div}$ denotes the divergence operator of $\partial\Omega$ and $X^T$ is the tangential part of the vector field $X.$
Since $\partial\Omega$ is closed, we obtain the desired identity from the divergence theorem.

Finally if, in particular, $k=0$, we integrate the identity $\Delta u=-1$ to get that
$c|\partial\Omega|=|\Omega|.$ Taking into account that $\langle X,\nu\rangle>0$ we conclude the desired result.

\end{proof}

The previous theorem provides a nonexistence result.

\begin{corollary}
    Let $M$ be a manifold endowed with a closed conformal vector field $X$ and constant scalar curvature $R=n(n-1)k$. Let $\Omega$ be a bounded domain with $\langle X,\nu\rangle>0$. If  $(n-1)-cnH$ has strict signal on $\partial\Omega$, then there exists no a solution of problem \ref{serrin}, on $\Omega$.
\end{corollary}

Theorem \ref{teoC} also provides a rigidity result to Serrin's problem. The following result uses this Minkowski-type formula (joined with Reilly's technique explored in Section 2) to obtain rigidity for geodesic balls. It is interesting to note that this coupling of techniques was used by Reilly himself to get an alternative proof for the classical Alexandrov's Theorem in the Euclidean space (see, for example, \cite{reilly} for this alternative proof, and also \cite{poggesi} for a comprehensive discussion relating this theorem with Serrin's problem in euclidean domains).


  \begin{corollary}\label{corolario4}
Let $(M^n,g)$ be an Einstein manifold with scalar curvature $R=n(n-1)k$ and endowed with a closed conformal vector field $X$. Let $\Omega$ be a bounded domain with $\langle X,\nu\rangle>0$ on $\partial\Omega$. If $u$ is a solution of \eqref{serrin}, then $\Omega$ is a geodesic ball, and $u$ is a radial function.
\end{corollary}
\begin{proof}
$P$-function \eqref{pfunction2} associated to \eqref{serrin} satisfies
\begin{equation}\label{final1}
P(u)_{\nu}=2{\nabla}^2 u(\nabla u,\nu)+\frac{2}{n}u_{\nu}+2kuu_{\nu}.   
\end{equation}
Furthermore,
\begin{eqnarray}\label{final2}
 Hu_{\nu}=|\nabla u|div \frac{\nabla u}{|\nabla u|}=\Delta u-\nabla^2 u(\nu,\nu)=-1-nku-\nabla^2 u(\nu,\nu).
\end{eqnarray}
Since $u=0$ on $\partial\Omega$, $\nabla u= u_{\nu}\nu$, we replace \eqref{final2} in \eqref{final1} to obtain that, along $\partial\Omega$, 
\begin{equation}\label{final3}
 P(u)_{\nu}=-\frac{2}{n}u_{\nu}((n-1)+nHu_{\nu}).   
\end{equation}
Lemma \ref{farina-roncoroni} implies that $P(u)$ is sub-harmonic. Known \eqref{final3}, Hopf's maximum principle applied to $P(u)$ gives
\begin{equation*}
-\frac{2}{n}u_{\nu}((n-1)+nHu_{\nu})\geq 0,   
\end{equation*}
on $\partial\Omega$. Thus, since $u_{\nu}=-c$ on $\partial\Omega$,
$$0\leq\int_{\partial\Omega}((n-1)+nHu_{\nu})\langle X,\nu\rangle.$$

From Theorem \ref{teoC} and above inequality  we conclude that $(n-1)+nHu_{\nu}=0.$ Thus, from Corollary \ref{meanc}, we get the desired result.
\end{proof}

By the end of this paper, we provide more one necessary condition to find solutions to Serrin's problem in manifolds with constant scalar curvature (here we are restricted to $k=0$ case in \eqref{serrin}, but a similar analysis also be done for any $k$ with a suitable change in the bound for the scalar curvature below).
\begin{theorem}
Let $M$ be a manifold with constant scalar curvature $R$ and endowed with a closed conformal vector field such that the conformal factor $\varphi$ of $X$ is positive. Suppose that $\Omega\subset M$ is a domain where a solution exists for 
\begin{equation*}\label{serrin,k=0}
\left\{\begin{array}{rcl}
\Delta u&=& -1,\\
u&>&  0\quad \hbox{in}\quad int(\Omega),\\
u&=& 0\quad \hbox{on}\quad  \partial\Omega,\\
|\nabla u|&=&c\quad \hbox{on}\quad  \partial\Omega,
\end{array}\right.
\end{equation*}
Then,
   $R>-\frac{|\partial\Omega|^2}{|\Omega|^2}\frac{(n-1)}{2(n-2)}\frac{(n+2)^2}{n}.$
\end{theorem}
\begin{proof}
    In fact, suppose by contradiction that $u$ is solution of \eqref{serrin} on a bounded domain $\Omega\subset M$
    satisfying $R\leq -\frac{|\partial\Omega|^2}{|\Omega|^2}\frac{(n-1)}{2(n-2)}\frac{(n+2)^2}{n}.$
    Since the scalar curvature is constant, we get, from \eqref{Pohozaev}, that 
    \begin{equation}\label{appl}
        \int_{\Omega}\varphi \left(\frac{n+2}{n}u-c^2+\frac{(n-2)R}{2n(n-1)}u^2\right)=0.
    \end{equation}
    
 It is easy to see that since $\Delta u= -1$, the divergence theorem implies that $c|\partial\Omega|=|\Omega|$. Thus, our constraint on the scalar curvature implies that 
 
 $$R\leq -\frac{n(n-1)}{2c^2(n-2)}\frac{(n+2)^2}{n^2},$$ which is equivalent to the inequality $\frac{(n+2)^2}{n^2}+2c^2\frac{(n-2)R}{n(n-1)}\leq 0.$ In particular, $\frac{(n-2)R}{n(n-1)}<0.$
Now, let us consider the quadratic function 
$y=\frac{n+2}{n}u-c^2+\frac{(n-2)R}{2n(n-1)}u^2.$ 
Note that the discriminant of the quadratic function is non-positive and, therefore, 
$\frac{n+2}{n}u-c^2+\frac{(n-2)R}{2n(n-1)}u^2\leq 0,$ where the equality occurs in a closed set. From \eqref{appl}, we reach a contradiction.
    
\end{proof}

\section*{Acknowledgements}
The authors would like to thank Alberto Farina and Luciano Mari for their discussions about the object of this paper and several valuable suggestions.

The first author would like to thank the hospitality of the Mathematics Department of  Università degli Studi di Torino, where part of this work was carried out (he was supported by CNPq/Brazil Grant 200261/2022-3). The first and third authors have been partially supported by Conselho Nacional de Desenvolvimento Científico e Tecnológico (CNPq) of the Ministry of Science, Technology and Innovation of Brazil, Grants 316080/2021-7 and 306524/2022-8, and supported by Paraíba State Research Foundation(FAPESQ), Brazil, Grant 3025/2021. The second author is member of GNAMPA, Gruppo Nazionale per l’Analisi Matematica, la Probabilità e le loro Applicazioni of INdAM.


\end{document}